\documentclass[12pt]{amsart}

\usepackage{amsrefs,amsmath}
\usepackage{amssymb,latexsym}
\usepackage{enumitem} 
\usepackage{comment}
\usepackage{aliascnt}
\usepackage[colorlinks=true]{hyperref}
\usepackage[mathscr]{euscript} 
\usepackage{fullpage}
\usepackage{supertabular}
\usepackage{multicol,multirow}
\usepackage{graphicx}
\usepackage{subeqnarray}

\makeatletter
\newcommand\Autoref[1]{\@first@ref#1,@}
\def\@throw@dot#1.#2@{#1}
\def\@set@refname#1{
    \edef\@tmp{\getrefbykeydefault{#1}{anchor}{}}%
    \def\@refname{\@nameuse{\expandafter\@throw@dot\@tmp.@autorefname}s}%
}
\def\@first@ref#1,#2{%
  \ifx#2@\autoref{#1}\let\@nextref\@gobble
  \else%
    \@set@refname{#1}
    \@refname~\ref{#1}
    \let\@nextref\@next@ref
  \fi%
  \@nextref#2%
}
\def\@next@ref#1,#2{%
   \ifx#2@ and~\ref{#1}\let\@nextref\@gobble
   \else, \ref{#1}
   \fi%
   \@nextref#2%
}
\makeatother

\newtheorem{thm}{Theorem}[section]
\newtheorem*{thm*}{Theorem}
\newaliascnt{prop}{thm}
\newtheorem{prop}[prop]{Proposition}
\aliascntresetthe{prop}
\newaliascnt{lem}{thm}
\newtheorem{lem}[lem]{Lemma}
\aliascntresetthe{lem}
\newaliascnt{cor}{thm}
\newtheorem{cor}[cor]{Corollary}
\aliascntresetthe{cor}

\theoremstyle{definition}
\newaliascnt{dfn}{thm}
\newtheorem{dfn}[dfn]{Definition}
\aliascntresetthe{dfn}
\newaliascnt{rmk}{thm}
\newtheorem{rmk}[rmk]{Remark}
\aliascntresetthe{rmk}
\newaliascnt{ex}{thm}
\newtheorem{ex}[ex]{Example}
\aliascntresetthe{ex}
\newaliascnt{ntn}{thm}

\aliascntresetthe{ntn}

\DeclareMathOperator\ann{Ann}
\DeclareMathOperator\Ass{Ass}

\DeclareMathOperator\Der{Der}
\DeclareMathOperator\Ext{Ext}
\DeclareMathOperator\Quot{Frac} 
\DeclareMathOperator\Ht{ht}
\DeclareMathOperator\lk{link} 
\DeclareMathOperator\rad{rad}
\DeclareMathOperator{\Hom}{Hom} 
\DeclareMathOperator{\Spec}{Spec} 
\DeclareMathOperator{\supp}{supp} 

\newcommand{\vect}[1]{\mathbf{#1}}
\newcommand{\D}{\mathscr D}
\newcommand{\m}{\mathfrak m}
\newcommand{\N}{\mathbb N}

\newcommand{\Z}{\mathbb Z}
\newcommand{\SRgen}[2]{\vect x^{\vect{#1}}\underline{\partial}^{\vect{#2}}}

\newcommand{\tns}{\otimes}

\title[Local Cohomology over Stanley-Reisner Rings]{Finiteness of Associated Primes of Local Cohomology Modules over Stanley-Reisner Rings}

\author{Roberto Barerra}
\address{
	Department of Mathematics, MSC \textup{470} \\
	Texas State University \\ 
	San Marcos, TX \textup{78666-4684} 
	}	
\email{rb1222@txstate.edu}	
\author{Jeffrey Madsen}
\address{
	Department of Mathematics \\
	Purdue University \\
	\textup{150} N. University Street 
	West Lafayette, IN \textup{47907-2067}
	}
\email{madsen4@purdue.edu}
\author{Ashley K. Wheeler}
\address{
	Department of Mathematics and Statistics, MSC \textup{1911} \\
	James Madison University \\ 
	Harrisonburg, VA \textup{22807}
	}
\email{wheeleak@jmu.edu}
	
\thanks{The authors wish to thank Wenliang Zhang, who supervised and directed our progress at the 2015 American Mathematical Society (AMS) Mathematics Research Communities (MRC).  The AMS MRC is funded by a grant from the National Science Foundation.  We are also grateful to Mark Johnson, Laura Matusevich, and David Eisenbud, who provided general feedback and ideas for follow up work.}

\begin{document}

\date{\today\; (last updated)}

\maketitle

\begin{abstract} 
Local cohomology modules, even over a Noetherian ring $R$, are typically unwieldy.  As such, it is of interest whether or not they have finitely many associated primes.  We prove the affirmative in the case where $R$ is a Stanley-Reisner ring over a field and its associated simplicial complex is a $T$-space.  
\end{abstract}

\section{Introduction}

All rings, unless otherwise stated, are commutative, Noetherian, and with 1.  Let $R$ denote a $K$-algebra, where $K$ is a field.  Let $I\subset R$ denote an ideal and $M$ an $R$-module.  We use the symbol $H^j_I(M)$, often suppressing the parentheses, to denote the $j$th local cohomology module over $M$ with support in $I$.  Local cohomology modules are generally not Noetherian -- even over a Noetherian ring -- and so they are not well-understood.  Finiteness of associated primes of local cohomology modules simplifies computations of invariants such as cohomological dimension of the Zariski space $\Spec(R)$.  In this paper, we prove that when $R$ is a Stanley-Reisner ring over $K$ and its associated simplicial complex is a $T$-space, the set $\Ass_R(H^j_IR)$ of associated primes of $H^j_I(R)$ is finite.  

\subsection{What is known} 

C. Huneke \cite{huneke92} first asked about finiteness of associated primes of local cohomology modules in the 1990s; around the time he and R. Sharp \cite{huneke+sharp} had affirmed it for $M=R$, a regular ring in characteristic $p>0$.  In 1993, G. Lyubeznik \cite{lyubeznik93} proved if $R$ is regular containing a field of characteristic $0$ then for any maximal ideal $\m$ of $R$, the number of associated primes for $H^j_I(R)$ that are contained in $\m$ is finite.  In contrast to Frobenius methods used in characteristic $p$, Lyubeznik used the then-burgeoning theory of $\D$-modules.  Using the theory of $F$-modules, Lyubeznik \cite{lyubeznik97} extended the characteristic $p$ results by removing the local hypothesis.  In an attempt to reconcile the $F$-module and $\D$-module methods, Lyubeznik \cite{lyubeznik00} also proved $H_I^j(R)$ has finitely many associated primes when $R$ is a regular local unramified ring in mixed characteristic.  Meanwhile, he was able to give an almost characteristic free proof \cite{lyubeznik00_1} for regular rings containing a field.  In 2011 he gave a characteristic-free proof of a result in the theory of $\D$-modules that reestablished the affirmative for the case of a polynomial ring over a field \cite{lyubeznik11}, characteristic-freely.    

\subsubsection{Counterexamples}

A. Singh \cite{singh} showed $|\Ass_R(H^3_IR|<\infty$ when $R=\Z[u,v,w,x,y,z]/(ux+vy+wz)$ and $I=(x,y,z)R$.  M. Katzman \cite{katzman} showed the set of associated primes for $H^2_{(x,y)}(R)$ is finite for $R=K[s,t,u,v,x,y]/(su^2x^2-(s+t)uxvy+tv^2y^2)$, where $K$ is any field.  Singh and I. Swanson \cite{singh+swanson} generalized Katzman's results with examples of normal hypersurfaces whose local cohomology modules have infinitely many associated primes.  

\subsubsection{Recent results}

The problem remains an active direction of research as only a few other, very special, affirmative cases are known.  H. Robbins \cite{robbins} answered Huneke's question in the affirmative for $M=R$, a polynomial ring or a power series ring over a two- or three-dimensional normal domain with an isolated singularity, finitely generated over a field of characteristic $0$; B. Bhatt, et al. \cite{bblsz} proved it when $M=R$ is a smooth $\Z$-algebra; S. Takagi and R. Takashashi \cite{takagi+takahashi} proved it for $M=R$ Gorenstein of finite $F$-representation type; T. Marley \cite{marley} proved the affirmative for $\dim{R}\leq 3$, $\dim{R}=4$ when $R$ is regular on the punctured spectrum, and $\dim{R}=5$ when $R$ is an unramified regular local ring and $M$ is finitely-generated and torsion-free; Marley and J. Vassilev \cite{marley+vassilev} proved $\Ass_R(H^j_IM)$ is finite when $R$ is a regular local ring and $\dim{M}\leq 3$, $\dim{R}\leq 4$, $\dim{M/IM}\leq 2$ and $M$ satisfies Serre's condition $S_{\dim{M}-3}$, or $\dim{R/I}=3$, $\ann_RM=0$, $R$ is unramified, and $M$ is $S_{\dim{M}-3}$; M. Brodmann and L. Faghani \cite{brodmann+faghani} showed $\Ass_R(H^j_IM)$ is finite if $M$ is finitely generated and $\Ass_R(H^1_IM),\dots,\Ass_R(H^{j-1}_IM)$ are all finite; M. Hellus \cite{hellus} proved the result when $M=R$ is local Cohen-Macaulay and either $\Ass_R(H^3_{(x,y)}R)$ is finite for every $x,y\in R$ or $\Ass(H^3_{(x_1,x_2,y)}R)$ is finite for $x_1,x_2\in R$ a regular sequence and $y\in R$. 
 
\subsection{Our results}

We use arguments from Lyubeznik \cites{lyubeznik00_1, lyubeznik11} to show the set of associated primes of local cohomology modules over a Stanley-Reisner ring $R=K[\Delta]$ (see \autoref{sec:Stanley-Reisner} for the notion of a Stanley-Reisner ring) is finite, provided the associated simplicial complex $\Delta$ is a $T$-space.  A simplicial complex $\Delta$ is a \emph{$T$-space} means either its face ideal is trivial or for all pairs of faces $F,G\in\Delta$ with $G\subsetneq F$, there exists a facet $H\subset\Delta$ such that $F\subseteq H$ but $G\not\subseteq H$.  If $\Delta$ is a $T$-space then we also say the Stanley-Reisner ring $R=K[\Delta]$ is a $T$-space.  

\begin{thm*}[cf. \autoref{cor:mainThm}]
Suppose $R=K[\Delta]$ is a $T$-space.  Let $I\subset R$ denote an ideal.  Then all local cohomology modules $H^j_I(R)$ have finitely many associated primes.
\end{thm*}

We follow a standard approach to proving finiteness of associated primes.  Given an $R$-module, $M$, and a filtration $[0=M_0\; \subset\;  M_1\; \subset\;  \cdots\; \subset\; M_l=M]$, such that each of the factors $M_j/M_{j-1}$ has finitely many associated primes; as long as $M$ has finite length we get the result $|\Ass_RM|<\infty$ by the containment
\[
\Ass(M)\subset \bigcup_j \Ass(M_j/M_{j-1}).
\]
Unfortunately, as mentioned before, it is not at all clear whether a local cohomology module has finite length!  Lyubeznik's strategy in \cite{lyubeznik93} uses a filtration in the appropriate $R$-algebra, namely the ring of $K$-linear differential operators of $R$, denoted $\D=\D_R=D(R;K)$ (see \autoref{sec:diffOps}).  Lyubeznik shows that local cohomology modules belong to the class of \textit{holonomic} $\D$-modules, and hence, have finite $\D$-length.  The key to our result is that the combinatorial construction of a $T$-space allows one to write down explicitly its ring of differential operators.  We then broaden Lyubeznik's redefinition \cite{lyubeznik11}*{Corollary 3.6} of V. Bavula's notion of holonomicity \cite{bavula} to include $T$-spaces.

\begin{thm*}[cf. \autoref{thm:xdelx}]
Suppose $R=K[\Delta]$ is a $T$-space. Then 
$\D=D(R;K)$ is generated as an $R$-algebra by elements of the form $x\partial^t$, where $x$ is one of the indeterminates in $R$ enumerating the vertices of the complex $\Delta$, $\partial^t$ denotes the divided power $\frac{1}{t!}\frac{\partial}{\partial x}$ (see \autoref{ex:weyl}), and $t\in \Z_{\geq 0}$.  
\end{thm*}


\begin{thm*}[cf. \autoref{thm:finitelength}]
Every holonomic $\D$-module, where holonomic is in the sense of \autoref{dfn:holonomicity}, has finite $\D$-length.
\end{thm*}

\begin{thm*} Suppose $R=K[\Delta]$ is a $T$-space.  Then:
\begin{itemize}
\item[(a) ](cf. \autoref{thm:Rholonomic}) $R$ is a holonomic $\D$-module.
\item[(b) ](cf. \autoref{thm:Rfholonomic}) The localized ring $R_f$ is a holonomic $\D$-module, for any $f\in R$.
\end{itemize}
\end{thm*}

\subsection{Notation, conventions, outline}

When the context is clear $\D=D(R;K)$ will denote the ring differential operators over $R$.  Denote the polynomial ring in $n$ variables over $K$ by $S=K[x_1,\dots,x_n]$.  $K[\Delta]=S/I_{\Delta}$ will denote a Stanley-Reisner ring with associated simplicial complex $\Delta$ on vertices labelled by the variables $x_1,\dots,x_n$.  For the sake of readibility, we may suppress indices in the indeterminates, for example writing $x^a$ to denote the arbitrary variable $x_i$ raised to the $a_i$th power.  We use multi-index notation $\SRgen{a}{t}=x_1^{a_1}\dots x_n^{a_n}\partial_1^{t_1}\dots \partial_n^{t_n}$ where $\partial_i^{t_i}=\frac{1}{t_i!}\frac{\partial_i^{t_i}}{\partial x_i^{t_i}}$ and again we may suppress indices when context is clear, writing $x\partial^t=x_i\partial^{t_i}_i$.  

The rest of the paper is structured as follows. \autoref{sec:preliminaries} consists entirely of preliminaries about local cohomology, $\D$-modules, and Stanley-Reisner rings. In \autoref{sec:result} we describe the $\D$-module structure of a Stanley-Reisner ring $R$ whose associated simplicial complex is a $T$-space, then we prove that for any $f\in R$, $R_f$ is holonomic. 

\section{Preliminaries}
\label{sec:preliminaries}

In this section, with the exception of \autoref{thm:primesFacets} we let $K$ denote a domain.

\subsection{Stanley-Reisner rings}
\label{sec:Stanley-Reisner}

Details surrounding many of the statements in this section can be found in Chapter 5 of \cite{bruns+herzog}.  Let $S=K[x_1,\dots,x_n]$ denote a polynomial ring over a domain $K$, and let $\Delta$ denote a simplicial complex on the vertex set, $V$, labeled by the indeterminates $x_1,\dots,x_n$; that is, $\Delta$ is a subset of the power set $2^V$, and if $F\in\Delta$ then so is every subset of $F$ (such is referred to as a \emph{slack} simplicial complex in \cite{brumatti+simis}).  The \emph{Stanley-Reisner ideal}, or \emph{face ideal}, of $\Delta$ over $K$ is the ideal in $S$ given by 
\[
I_\Delta := (x_{i_1}\cdots x_{i_s} \mid \{x_{i_1},\ldots,x_{i_s}\}\notin \Delta).
\]
The \emph{Stanley-Reisner ring} of $\Delta$ over $K$ is the quotient ring $K[\Delta]:=S/I_\Delta$.  Without loss of generality we shall assert all singletons $\{x_i\}$ belong to $\Delta$, or else omit the missing, or \emph{slack}, vertices and work over the resulting subcomplex which we denote by $\tilde{\Delta}$.  If $x_{i_1},\dots,x_{i_s}$ are the slack vertices we have       
\[
I_{\Delta}=(x_{i_1},\dots x_{i_s})+I_{\tilde{\Delta}}S.
\]

Sets in $\Delta$ are called \emph{faces}.  Maximal faces with respect to inclusion in $\Delta$ are called \emph{facets}.  Each face $F=\{x_{i_1},\dots,x_{i_s}\}$ corresponds to a monomial $\mu=x_{i_1}\cdots x_{i_s}\in S$ and we write $F=F_{\mu}$.  Likewise, every square-free monic monomial $\mu\notin I_{\Delta}$ corresponds to a face $F_{\mu}\in\Delta$.  Given a vector $\vect t=(t_1,\dots,t_n)\in\N^n$, we say the \emph{support} of $\vect t$ is the set 
\[
\supp(\vect t):=\{x_{i_1},\dots,x_{i_s}\mid t_j\neq 0 \text{ if and only if } j=i_1,\dots,i_s\}.
\]

\begin{thm}[cf. \cite{bruns+herzog}*{Theorem 5.1.4}]
\label{thm:primesFacets}
Suppose $R=K[\Delta]$ is a Stanley-Reisner ring over a field $K$.  Then the minimal primes of the face ideal $I_{\Delta}$ are in bijection with the facets of $\Delta$.  That is, if $P=(x_{i_1},\dots,x_{i_s})$ is a minimal prime of $I_{\Delta}$ then the corresponding facet $F=F_P$ is given by the complementary vertices $F=\{x_j\mid j \neq i_1,\dots,i_s\}$; and if $F=\{x_{i_1},\dots, x_{i_s}\}$ is a facet in $\Delta$ then $P=P_F$ is the minimal prime given by the complementary variables $P=(x_j\mid j\neq i_1,\dots i_s)$.     
\end{thm}

A Stanley-Reisner ring $R=K[\Delta]$ is a homogeneous $K$-algebra and so has a \emph{Hilbert function} $H(R,j):=\dim_KR_j$, the length as a $K$-module of the $j$th graded component of $R$.  The Hilbert function of a Stanley-Reisner ring actually coincides with the Hilbert polynomial for $j>0$ (cf. \cite{bruns+herzog}*{Theorem 5.1.7}). 

\begin{ex}[Tripp; cf. \cite{tripp}*{Section 6}]
\label{ex:trippRing} 
Let $\Delta$ denote the complex in \autoref{fig:Tspace}, whose facets are given by $\{x,y\},\{x,z\},\{y,z\},\{w\}\subset V=\{x,y,z,w\}$; i.e., the union of $w$ with the boundary of the $2$-simplex spanned by $x,y,z$. 
\begin{centering}
\begin{figure}
\includegraphics[scale=0.2]{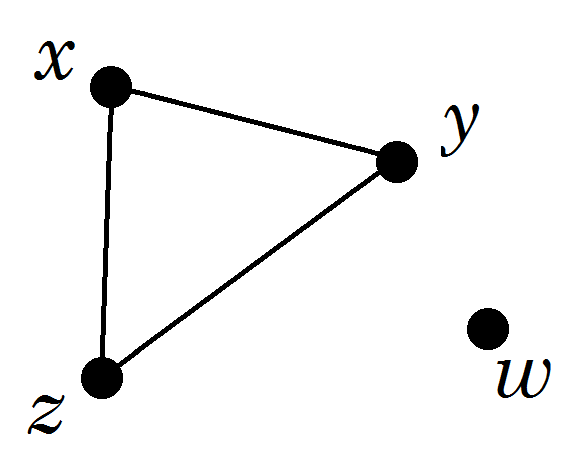}
\caption{$T$-space in \autoref{ex:trippRing}.}
\label{fig:Tspace}
\end{figure}
\end{centering}
Then we have 
\[
\Delta=\{\{x,y\},\{x,z\},\{y,z\},\{w\},\{x\},\{y\},\{z\},\emptyset\}\quad \text{and}\quad I_\Delta=(xw,yw,zw,xyz).
\]
The minimal primes for $R=K[\Delta]$ are complementary to the facets: 
\[
I_{\Delta}=(x,w)\cap(y,w)\cap(z,w)\cap(x,y,z).
\]
\end{ex}

A face $F\in\Delta$ may be \emph{separated} from another face $G\in\Delta$ means there exists a facet containing $F$ but not containing $G$.  We remark that the containment is not necessarily strict.  A complex $\Delta$ is a \emph{$T$-space} means for every face $F\in\Delta$, if $G\in \Delta$ is not contained in $F$, then $F$ can be separated from $G$.  In \autoref{ex:trippRing}, one can directly check that $\Delta$ is a $T$-space. 

\begin{prop}[cf. \cite{tripp}*{5.3 Lemma [sic]}]
\label{prop:Tspace}
Suppose $\Delta$ is a simplicial complex, and its face ideal $I_{\Delta}\subset S$ is non-trivial.  Then $\Delta$ is a $T$-space if and only if $F$ may be separated from $\{v\}$ for all faces $F\in\Delta$ and vertices $v\notin F$.  
\end{prop}

\begin{ex}[cf. \cite{tripp}*{Section 6}]
\label{ex:graph}
The complex in \autoref{fig:notTspace}, given by facets $\{x,y\}$ and $\{z,w\}$, is not a $T$-space.  The reason is that $\{x\}$ and $\{y\}$ cannot be separated from one another and nor can $\{z\}$ and $\{w\}$.  

\begin{centering}
\begin{figure}
\includegraphics[scale=0.2]{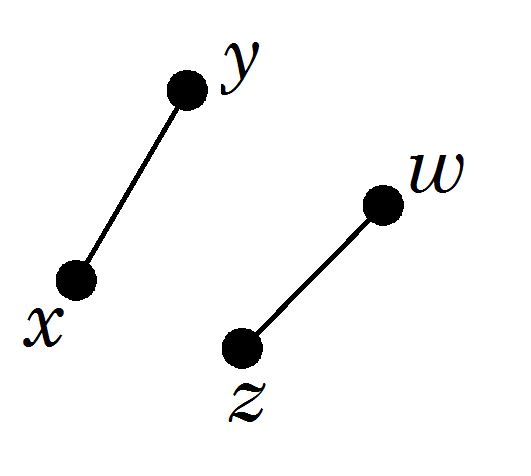}
\caption{This graph (see \autoref{ex:graph}) is not a $T$-space.} 
\label{fig:notTspace}
\end{figure}
\end{centering}

In fact, if $\Delta$ is a graph (i.e., no facets have dimension more than $1$) then it is a $T$-space if and only if none of its vertices have degree $1$.  To see why, first suppose $\Delta$ has a vertex $v$ of degree $1$.  Then the unique facet $H$ containing $v$ is an edge.  If $w$ denotes the vertex adjacent to $v$ then $\{v\}\cap\{w\}=\emptyset$ but $w\in H$.  Thus $\{v\}$ cannot be separated from $\{w\}$.  On the other hand, suppose $v\in \Delta$ and no vertices in $\Delta$ have degree $1$.  We may assert no vertex has degree $0$, either, since such a vertex can be separated from any other.  Thus all facets in $\Delta$ are edges, while the degree condition ensures every vertex is contained in at least two facets.  The only way two vertices will share more than one edge, meaning they are contained in at least two common facets, is if $\Delta$ is a multigraph.  However, that cannot happen since $\Delta$ is a simplicial complex.  
\end{ex}

\begin{ex}
\label{ex:link}
The \emph{link} of a face $F\in\Delta$ is the set 
\[
\lk_{\Delta}(F):=\left\{G\in\Delta\mid F\cup G\in\Delta, F\cap G=\emptyset \right\}.
\]
Using \autoref{prop:Tspace}, it is clear that if $\Delta$ is a $T$-space then so is $\lk_{\Delta}F$ for any face $F\in\Delta$.
\end{ex}

\subsection{Local cohomology and $\D$-modules}

%
Let $I$ be an ideal of a Noetherian ring $R$ and let $M$ be an $R$-module, which may or may not be Noetherian.  The \emph{$j$th local cohomology module over $M$ with support in $I$} is defined as the following direct limit of $\Ext$ modules, 
\[
H^j_I(M):=\varinjlim_t\,\Ext^j_R(R/I^t,M).
\]
It is the right derived functor of $H^0_I(?)$, where $H^0_I(M):=\cup_t\ann_MI^t = \{u\in M\mid uI^t=0\} = \varinjlim_t\Hom_R(R/I^t,M)$, the global sections of the sheaf $\tilde M$ with support on the closed subscheme $\Spec{R/I}\subset \Spec R$.  

Suppose $f_1,\dots,f_s\in R$ and $\rad\,I=\rad(f_1,\dots,f_s)R$.  In practice, we compute local cohomology by using the direct limit of the Kozsul complexes $\mathcal K^i(f_1^t,\dots,f_s^t;M)$, which is the \u Cech complex:
\begin{equation}
\label{eqn:cechComplex}
0\to M\to \bigoplus_iM_{f_i}\to \bigoplus_{i<j}M_{f_if_j}\to \cdots \to M_{f_1\cdots f_s}\to 0.
\end{equation}
The local cohomology modules are the cohomology modules of the \u Cech complex.  

\subsubsection{The Ring of Differential Operators}
\label{sec:diffOps}

One can read more about what follows in any of \cites{lyubeznik00_1,traves,24hours}.  Let $K$ denote a domain and let $R$ denote a $K$-algebra of finite type.  The \emph{$K$-linear differential operators on $R$ of order $0$} are the multiplication maps $R\xrightarrow{\cdot r} R$ for $r\in R$.  For a positive integer $t$, a \emph{$K$-linear differential operator on $R$ of order $\leq t$} is a $K$-linear map $\delta:R\to R$ such that for any $r\in R$, as an operator of order $0$, the commutator $\left[\delta,r\right]=\delta r-r\delta \left(=\delta\circ r-r\circ\delta\right)$ is a $K$-linear operator of order $\leq t-1$.  For example, the module $\Der_K(R,R)$ of $K$-derivations on $R$ is $D^1(R;K)$, also the $R$-dual of the module $\Omega_{R/K}$ of K\"ahler differentials.

The union of $K$-linear differential operators of all orders is known as the \emph{ring of differential operators} on $R$, $D(R;K)$, which we denote $\D$, or $\D_R$, if more clarity is necessary.  The multiplication operation in $\D$ is composition.  The $R$-modules $D^t(R;K)$, given by the operators of order $\leq t$, give $\D$ a graded structure; identifying elements in $R$ with themselves as operators in $D^0(R;K)$ makes $\D$ into an $R$-algebra.  

\begin{ex}
Let $K=\mathbb C$, the field of complex numbers, and let $R=\mathbb C[x]$ denote the polynomial ring in one variable.  The usual differential operator $\delta=\frac{\partial}{\partial x}$ is a $\mathbb C$-derivation on $R$, and  
\[
D(\mathbb C[x];\mathbb C)=\frac{\mathbb C\langle x,\frac{\partial}{\partial x}\rangle}{\left(\frac{\partial}{\partial x}\cdot x-x\cdot\frac{\partial}{\partial x}-1\right)\mathbb C\langle x,\frac{\partial}{\partial x}\rangle}.
\]
\end{ex}

\begin{ex} 
\label{ex:weyl}
More generally, let $R=S=K[x_1,\dots,x_n]$, the polynomial ring.  If $K$ is a field of characteristic $0$ then $\D$ is the first Weyl algebra.  If $K$ is a field of characteristic $p>0$ then $\D$ is strictly larger than the Weyl algebra, as it includes the \emph{divided powers} $\frac{1}{p!}\frac{\partial^p}{\partial x_i^p}$.  For any characteristic of $K$ and $t\geq 0$, we write $\partial_i^t$ to denote the $K[x_1,\dots,\hat{x_i},\dots,x_n]$-linear map (the hat connotes omission of the variable $x_i$)
\begin{subeqnarray*}
\frac{1}{t!}\frac{\partial^t}{\partial x_i^t}:S &\to& S \\
x_i^v &\mapsto& \binom{v}{t}x_i^{v-t}.
\end{subeqnarray*}
\end{ex}

\begin{ex}
\label{ex:idealizer} 
If $R=S/J$ for some ideal $J\subset S$ and $K$ is a field then 
\[
\D_R\cong\frac{\left\{ \delta\in \D_S \mid \delta (J)\subset J\right\}}{J\D_S}.
\]
The set $\mathscr I(J):=\left\{\delta\in \D_S \mid \delta(J) \subset J\right\}$ is called the \emph{idealizer} of $J$. 
\end{ex}

Note, $\D$ is generally non-commutative.  We say ``$\D$-module" to mean a left $\D$-module.  

\subsubsection{Localization}
\label{sec:localization}

Let $K$ denote a domain and $R$ a $K$-algebra of finite type.  Let $C(R)$ denote the smallest subcategory of $\D$-modules containing all modules of the form $R_f$, with $f\in R$, that is closed under formation of subobjects, quotient objects, and extensions.  Let $C_0(R)\subset C(R)$ denote the full subcategory whose objects are the modules $R_f$, for any $f\in R$.  For integers $s>0$, let $C_s(R)$ denote the subobjects, quotient objects, and extensions of the objects in $C_{s-1}(R)$.  Then $C(R)$ is the directed union of the subcategories $C_s(R)$.  The following is in \cite{lyubeznik00_1}, but in the case where $R$ is regular.     

\begin{prop}[cf. \cite{lyubeznik00_1}*{Lemma 5}] Suppose $M\in C(R)$.  Then:
\begin{itemize}[leftmargin=20pt]
\item[(a) ] $M_f\in C(R)$, for all $f\in R$.
\item[(b) ] $H_I^j(M)\in C(R)$, for any ideal $I\subset R$ and $j\geq 0$.  
\end{itemize}
\end{prop}

\begin{proof}
There exists some $s$ such that $M\in C_s(R)$.  If $s=0$, then there exists $g\in R$ where $M=R_g$ and so $M_f=R_{gf}\in C_0(R)$.  If $s>0$ then $M$ is a subobject, quotient object, or extension of a module or modules in $C_{s-1}(R)$. In the case $M$ is a subobject or quotient object of a module $M'\in C_{s-1}(R)$, $M_f$ is respectively a subobject or quotient object of $M'_f$.  By induction on $s$, $M'_f\in C_{s-1}(R)$, so $M_f\in C_s(R)$.  In the case where $M$ is an extension of objects $M',M''\in C_{s-1}(R)$, $M_f$ is an extension of $M'_f$ and $M''_f$, which again are in $C_{s-1}(R)$ by induction, and so $M_f\in C_s(R)$, thereby proving (a). 

To prove (b) it suffices to show we have either $H^j_I(M)\in C(R)$, or else the kernel, image, or cokernel of a map from the long exact sequence of $\D$-modules 
\begin{subeqnarray*}
0\to H^0_{I'}(M)\to H^0_{I''}(M)\to H^0_I(M) 
\to H^1_{I'}(M)\to H^1_{I''}(M)\to H^1_I(M)
\to \cdots
\end{subeqnarray*}
where $I''\subset I'$ are ideals in $R$, is in $C(R)$.  

By construction $C(R)$ is closed under formation of subobjects, quotient objects, and extensions so in fact it suffices to only prove the statement for $H_I^j(M)$.  Write $I=(f_1,\dots,f_s)$; the local cohomology modules $H_I^j(M)$ are the cohomology modules of the \u Cech complex \autoref{eqn:cechComplex}, whose terms are, by hypothesis, in $C(R)$.  Therefore so are the modules $H^j_I(M)$.
\end{proof}

\begin{cor}[cf. \cite{lyubeznik00_1}*{Corollaries 6-7}]
\label{cor:finiteLengthObjects}
Suppose $R_f$ has finite length in the category of $\D$-modules, for all $f\in R$, and suppose $M\in C(R)$.  Then:
\begin{itemize}[leftmargin=20pt]
\item[(a) ] $M$ has finite $\D$-length.
\item[(b) ] $M$ has finitely many associated primes.
\end{itemize}
\end{cor}

\begin{proof}
By construction of $C(R)$ there exists $s$ such that $M\in C_s(R)$. We use induction to prove (a); if $s=0$, then $M=R_f$ and the hypothesis is the conclusion.  If $s>0$ then there are two situations to check.  In the case $M$ is a subobject or quotient of an object $M'\in C_{s-1}(R)$, its length is clearly bounded by the length of $M'$, and the length of $M'$ is finite by the inductive hypothesis.  The other situation is where $M$ is an extension of two objects $M',M''\in C_{s-1}(R)$, in which case its length is bounded by the sum of the, again, finite, lengths of $M'$ and $M''$.

To prove (b), given a filtration $[0=M_0\subset M_1\subset\cdots\subset M_i\subset\cdots\subset M]$ we have the containment of sets of associated primes
\begin{equation}
\label{eq:assM}
\Ass_R{M}\subset\bigcup_j \Ass_R{(M_j/M_{j-1})}.
\end{equation}
We shall construct a finite filtration where each of the factors $M_j/M_{j-1}$ has finitely many associated primes.  Let $P_1$ denote a maximal element in $\Ass{M}$.  Then $H^0_{P_1}(M)$ is a non-zero $\D$-submodule of $M$ with exactly one associated prime, $P_1$.  Set $M_1=H^0_{P_1}(M)$ and let $P_2$ be a maximal element in $\Ass{(M/M_1)}$.  Then $H^0_{P_2}{(M/M_1)}$ is non-zero and has exactly one associated prime, $P_2$.  Let $M_2$ denote the preimage of $H^0_{P_1}{(M/M_1)}$ in $M$.  We continue this recursive process to construct a filtration that must stabilize, since by part (a) $M$ has finite length as a $\D$-module.  Then \autoref{eq:assM} bounds the number of elements in $\Ass{M}$.
\end{proof}

\section{The Main Result}
\label{sec:result}

Unless otherwise specified, $R=K[\Delta]=S/I_{\Delta}$ is a Stanley-Reisner ring over a field $K$ and $S=K[x_1,\dots,x_n]$, as usual, is the polynomial ring in $n$ variables over $K$.  We let $\D_S=D(S;K)$ and $\D=\D_R=D(R;K)$.  As a result of the isomorphism in \autoref{ex:idealizer} we refer to a monomial $\vect x^{\vect a}\underline{\partial}^{\vect t}\in \D_S$ as ``in" $\D_R$ if its image in the idealizer of $I_{\Delta}$ is non-zero.   

The heart of our paper will be in proving $R_f$ has finite length for any $f\in R$.  Historically this was shown for $S$ using the notion of holonomicity, by J. Bernstein \cite{bernstein+gelfand+gelfand} in characteristic $0$, and by R. B\o gvad \cite{bogvad} in characteristic $p>0$.  Bavula \cite{bavula} gives a characteristic-free definition of holonomicity which Lyubeznik \cite{lyubeznik11} has further simplified.  In \autoref{sec:holonomicity} we show how Lyubeznik's definition applies to Stanley-Reisner rings.  We show that $R_f$ is holonomic under this definition, and has finite length.  It will then follow from \autoref{cor:finiteLengthObjects} that $|\Ass_R(H_I^jR)|<\infty$.

\subsection{The \texorpdfstring{$\D$}{D}-module structure of a $T$-space}
\label{sec:DstructureForSR}

In \cite{traves} W. Traves showed that a monomial $\vect x^{\vect a}\underline{\partial}^{\vect t}\in \D_S$ is in $\D_R$ if and only if for each minimal prime $P$ of $I_{\Delta}$, either $\vect x^{\vect a}\in P$ or $\vect x^{\vect t}\notin P$.  The $T$-space condition adds enough information to completely characterize $\D_R$.  
  
\begin{thm}
\label{thm:xdelx}
Suppose $R=K[\Delta]$ is a $T$-space.  Then $\D$ is generated as an $R$-algebra by operators of the form $x\partial^t$, where $t\in \N$.  In other words,
\[
\D=R\left\langle x_i\partial_i^{t_i}\mid 1\le i\le n, t_i\ge 0\right\rangle.
\]
\end{thm}

\begin{proof}
By \cite{traves}*{Theorem 3.5}, $\D$ is generated as a $K$-module by operators $\SRgen{a}{t}$ such that for every minimal prime $P$ of $R$, either $\vect x^{\vect a}\in P$ or $\vect x^{\vect t}\not\in P$.  Elements $x\partial^t$ satisfy this condition tautologically.  For the other inclusion, suppose $\SRgen{a}{t}\in\D$.  Put $F=\supp(\vect a)$ and $G=\supp(\vect t)$.  We assert $F\in\Delta$, for otherwise $\vect x^{\vect a}\in I_{\Delta}$.  Choose a facet $H$ containing $F$ and let $P_H$ denote its corresponding minimal prime.  Then $\vect x^{\vect a}\notin P_H$ by \autoref{thm:primesFacets}, and hence, by \cite{traves}*{Theorem 3.5}, $\vect x^{\vect t}\notin P_H$.  It follows that $G\subseteq H$, and so $F$ cannot be separated from $G$.  By the $T$-space condition we must then have $G\subseteq F$, i.e., $\supp(\vect t)\subseteq \supp(\vect a)$.  Therefore, monomials of the form $x\partial^{t}$ divide $\SRgen{a}{t}$.
\end{proof}

%
%

The $R$-algebra generators of $\D$ are monomials in which for each $i=1,\dots, n$, a power of $\partial_i$ does not appear without $x_i$.  It follows from \cite{traves} that a free (left) $K$-basis for $\D$ is given by monomials of the form $\SRgen{a}{t}$, such that $\supp(\vect t)\subseteq \supp(\vect a)\in \Delta$.  

In general, if $f$ and $g$ are polynomials in the variable $x$, then the higher order product rule says
\begin{equation}
\label{eq:prodRule}
\frac{\partial^t}{\partial x^t}(fg)=\sum_{s=0}^t\binom{t}{s}\frac{\partial^s}{\partial x^s}(f)\frac{\partial^{t-s}}{\partial x^{t-s}}(g).
\end{equation}
Multiplying \autoref{eq:prodRule} on the left by $x$ and dividing by $t!$, it follows that for every $f\in R$ we have the following relation in $\D$:
\begin{equation}
\label{eq:opProdRule}
x\partial^t\cdot f = \sum_{s=0}^t\underbrace{\frac{1}{s!}\frac{\partial^s}{\partial x^s}(f)}_{\in R=D^0(R;K)}\cdot\ x\partial^{t-s}.
\end{equation}

\begin{prop}[cf. \cite{lyubeznik11}*{Corollary 2.2}]
\label{prop:xdxcomm}
Let $R=K[\Delta]$ be a $T$-space.  Then
\begin{itemize}
\item[(a) ] $x\partial^t\cdot x^u=\sum_{s=0}^t\binom{u}{s}x^{u-s+1}\partial^{t-s}$ for $u\geq 0$.  
\item[(b) ] $x_i\partial_i^{t_i}$ commutes with $x_j$ for $j\neq i$ and with $x_j\partial_j^{t_j}$ for any $j$.

\end{itemize}
\end{prop}

\begin{proof}
For (a), apply \autoref{eq:opProdRule} to $f=x^u$.  To prove (b), take $f\in R$.  By applying \autoref{eq:prodRule} to $\frac{\partial^{t_i}}{\partial x_i^{t_i}}\left(x_j\cdot \frac{1}{t_j!}\frac{\partial^{t_j}}{\partial x_j^{t_j}}(f)\right)$, for $t_j\geq 0$,
\begin{subeqnarray*}
x_i\partial_i^{t_i}\cdot x_j\partial_j^{t_j}(f) &=& x_i\frac{1}{t_i!}\frac{\partial^{t_i}}{\partial x_i^{t_i}}\left(x_j\cdot \frac{1}{t_j!}\frac{\partial^{t_j}}{\partial x_j^{t_j}}(f)\right) \\
	&=& x_i\frac{1}{t_i!}\sum_{s=0}^{t_i}\binom{t_i}{s}\frac{\partial^s}{\partial x_i^s}(x_j)\frac{\partial^{t_i-s}}{\partial x_i^{t_i-s}}\left(\frac{1}{t_j!}\frac{\partial^{t_j}}{\partial x_j^{t_j}}(f)\right) \\
	&=& x_i\frac{1}{t_i!}x_j\frac{\partial^{t_i}}{\partial x_i^{t_i}}\left(\frac{1}{t_j!}\frac{\partial^{t_j}}{\partial x_j^{t_j}}(f)\right) \\
	&=& \frac{1}{t_i!t_j!}x_ix_j\frac{\partial^{t_i}\partial^{t_j}}{\partial x_i^{t_i}\partial x_j^{t_j}}(f).
\end{subeqnarray*}
Interchanging $i$ and $j$ gives the same result.  For $j\neq i$, setting $t_j=0$ shows that $x_i\partial_i^{t_i}$ commutes with $x_j$. 
\end{proof}

\begin{prop}
\label{prop:Daction}
For any $f,g\in R$, $j\geq 0$, 
\begin{equation}
\label{eq:quotientRule}
x\partial^t\left(\frac{g}{f^j}\right)
=\frac{1}{f^j}x\partial^t(g)
-\sum_{s=1}^t\frac{1}{f^j}\frac{1}{s!}\frac{\partial^sf^j}{\partial x^s}\cdot x\partial^{t-s}\left(\frac{g}{f^j}\right)
\end{equation}
with $x\partial^0=x$.
\end{prop}

\begin{proof}
$R_f$ acquires a structure of a $\D$-module as follows.  By \autoref{eq:opProdRule}, 
\begin{subeqnarray*}
x\partial^t\cdot f &=& \sum_{s=0}^t\frac{1}{s!}\frac{\partial^sf}{\partial x^s}\cdot x\partial^{t-s} \\
x\partial^t\cdot f-\sum_{s=1}^t\frac{1}{s!}\frac{\partial^sf}{\partial x^s}\cdot x\partial^{t-s} &=& \partial^0(f)\cdot x\partial^t \\
	&=& f\cdot x\partial^t. 
\end{subeqnarray*}
We transpose the equality, replace $f$ with $f^j$, then multiply on the left by $\frac{1}{f^j}$ to get
\begin{equation}
\label{eq:operator}
x\partial^t = \frac{1}{f^j}x\partial^t\cdot f^j -\sum_{s=1}^t\frac{1}{f^j}\frac{1}{s!}\frac{\partial^sf^j}{\partial x^s}\cdot x\partial^{t-s}. 
\end{equation}
Finally, apply the operator \autoref{eq:operator} to $\frac{g}{f^j}\in R_f$.
\end{proof}

\begin{rmk}
\label{rmk:Daction} 
\autoref{prop:Daction} also holds for any $\D$-module $M$, with $g\in R$ replaced by $u\in M$.  In other words, $M_f$ likewise inherits a $\D$-module structure from $M$.
\end{rmk}

\subsection{Holonomicity}
\label{sec:holonomicity}

\begin{dfn}[cf. \cites{lyubeznik11,bavula}]
\label{dfn:holonomicity}
Let $S=K[x_1,\dots,x_n]$ where $K$ is a domain.  Let $R=S/J$ denote a homomorphic image of $S$ and let $r=\dim R$.
\begin{itemize}
\item[(a) ] The \emph{Bernstein filtration} on $\D$ is given by $\mathscr F=[\mathscr F_0\subset \mathscr F_1\subset \cdots \subset \D]$ where 
\[
\mathscr{F}_j=K\cdot \{\vect x^\mathbf{a}\underline{\partial}^{\vect t}\mid a_1+\cdots +a_n+t_1+\cdots+t_n\leq j\}\cap \mathscr I(J)/J\D_S
\]
(the symbol $\mathscr I(J)$ refers to the idealizer of $J$; see \autoref{ex:idealizer}).
\item[(b) ] A \emph{$K$-filtration} $\mathscr G$ on a $\D$-module $M$ is an ascending chain of $K$-modules $\mathscr G_0\subset \mathscr G_1\subset\cdots$
such that
\begin{itemize}
\item[(i) ] $\cup_i\mathscr G_i=M$ and
\item[(ii) ] for all $i$ and $j$, $\mathscr F_j\mathscr G_i\subset \mathscr G_{i+j}$.
\end{itemize}
\item[(c) ] A $\D$-module $M$ is \emph{holonomic} means it has a $K$-filtration $\mathscr G$ compatible with $\mathscr F$ such that for all $i$, $\dim_K{\mathscr G_i}\leq Ci^r$, for some constant $C$.
\end{itemize}
\end{dfn}

In the following Lemma, we will use the notation
\[
\SRgen{}{t}=\SRgen{\sum_{x_{\it i}\in \supp(t)}\vect e_{\it i}}{t},
\]       
where $\vect e_i$ denotes the $i$th standard unit vector in $\N^n$.  

\begin{lem}
\label{lem:ddm}
Let $R=K[\Delta]$ be a $T$-space and let $\m=(x_1-c_1,\ldots,x_n-c_n)R$ be a $K$-rational maximal ideal.
\begin{itemize}
\item[(a) ] A $K$-basis for $\D/\D\m$ is given by $\{\SRgen{}{t}\mid \supp{(\vect t)}\in\Delta\}$.
\item[(b) ] For any $w\in\D/\D\m$ with $w\notin K$, there exists $f\in\m$ such that $fw=1$.
\end{itemize}
\end{lem}

\begin{proof}
To show (a), note $\D$ is generated as a $K$-vector space by elements of the form $\vect x^{\vect a}\SRgen{}{t}$, with $\supp(\vect t)\in\Delta$. We shall determine the relations among these generators in $\D/\D\m$. Consider the element $\vect x^{\vect a-\vect e_i}\SRgen{}{t}(x_i-c_i)\in \D\m$, where $a_i\geq 1$.  Note, elements of this form generate $\D\m$ as a $K$-vector space because equivalently we could consider $\vect x^{\vect a}\SRgen{}{t}(x_i-c_i)$ with $a_i\geq 0$.  Using \autoref{prop:xdxcomm}, 
\begin{subeqnarray*}
\vect x^{\vect a-\vect e_i}\SRgen{}{t}(x_i-c_i) 
	&=& \vect x^{\vect a -\vect e_i}\SRgen{}{t} x_i
		-c_i\vect x^{\vect a-\vect e_i}\SRgen{}{t} \\
	&=& \vect x^{\vect a-\vect e_i}\SRgen{}{t-\it{t_i}\vect e_{\it i}}x_i\partial_i^{t_i} x_i
		-c_i\vect x^{\vect a-\vect e_i}\SRgen{}{t} \\
	&=& \vect x^{\vect a-\vect e_i}\SRgen{}{t-\it{t_i}\vect e_{\it i}}x_i^2\partial_i^{t_i} 
		+\vect x^{\vect a-\vect e_i}\SRgen{}{t-\it{t_i}\vect e_{\it i}}x_i\partial_i^{t_i-1} 
		-c_i\vect x^{\vect a-\vect e_i}\SRgen{}{t} \\
	&=& \begin{cases}
		\vect x^{\vect a}\SRgen{}{t} 
			+\vect x^{\vect a-\vect e_i}\SRgen{}{t-\vect e_{\it i}} 
			-c_i\vect x^{\vect a-\vect e_i}\SRgen{}{t}
			& t_i\geq 2 \\
		\vect x^{\vect a}\SRgen{}{t}
			+\vect x^{\vect a}\SRgen{}{t-e_{\it i}} 
			-c_i\vect x^{\vect a-\vect e_i}\SRgen{}{t}
			& t_i=1 \\
		\vect x^{\vect a}\SRgen{}{t}
			-c_i\vect x^{\vect a-\vect e_i}\SRgen{}{t}
			& t_i=0
		\end{cases}
\end{subeqnarray*}
giving the following congruence in $\D/\D\m$:
\begin{equation}
\label{eq:cong}
\vect x^{\vect a}\SRgen{}{t}\equiv\begin{cases}	
	c_i\vect x^{\vect a-\vect e_i}\SRgen{}{t}-\vect x^{\vect a-\vect e_i}\SRgen{}{t-e_{\it i}} & t_i\geq 2 \\
	c_i\vect x^{\vect a-\vect e_i}\SRgen{}{t}-\vect x^{\vect a}\SRgen{}{t-e_{\it i}} & t_i=1 \\
	c_i\vect x^{\vect a-\vect e_i}\SRgen{}{t} & t_i=0
\end{cases}.	 
\end{equation}
Applying \autoref{eq:cong} repeatedly allows us to uniquely reduce $\vect x^{\vect a}\SRgen{}{t}$ to a $K$-linear sum of elements in $\{\SRgen{}{t}\mid \supp{(\vect t)}\in\Delta\}$.

To prove (b), first observe applying \autoref{eq:cong} to the case of $\vect{a}=a_i\vect e_i$ while inducing on $a_i\geq 1$,
\[
(x_i-c_i)^{a_i}\SRgen{}{t}
	\equiv\begin{cases}
		-\SRgen{}{t-{\it a_i}e_{\it i}} & t_i\geq a_i+1 \\
		-x_i\SRgen{}{t-{\it a_i}e_{\it i}}
			\equiv-c_i\SRgen{}{t-{\it a_i}e_{\it i}} & t_i=a_i \\
		0 & t_i\leq a_i-1
\end{cases}
\]
shows for any $\vect a=(a_1,\dots,a_n)$ and setting $\vect c=(c_1,\dots,c_n)\in K^n$, $(\vect x-\vect c)^{\vect a}$ annihilates $\SRgen{}{\vect t}$ if $a_i>t_i$ for some $i$.  Meanwhile, $(\vect x-\vect c)^{\vect t}\SRgen{}{t}\in K^{\times}=K\setminus 0$.  

Write $w=\sum_{j=1}^sb_j\SRgen{}{\vect t_{\it j}}\in \D/\D\m$, asserting $b_j\in K^{\times}$ for all $j=1,\dots,s$ and $s\geq 1$.  Take any $l$ such that $\vect t_l$ is maximal in $\{\vect t_j\mid j=1,\dots,s\}$ under the partial order on $\mathbb{N}^n$, and let $g=(\vect x-\vect c)^{\vect t_j}$.  By our choice of $l$, for any $j\neq l$, $g$ annihilates $\SRgen{}{\vect t_{\it j}}$. Thus $gw=b_l(g\SRgen{}{\vect t_{\it l}})$, and $g\SRgen{}{\vect t_{\it l}}$ is a unit by our observations above.  Put $f=(gw)^{-1}g$ to get $fw=1$.
\end{proof}

\begin{prop}[cf. \cite{lyubeznik11}*{Corollary 2.5}]
\label{prop:ddmbasis}
Suppose $R=K[\Delta]$ is a $T$-space.  Let $\m\subset R$ denote a maximal ideal, such that $R/\m$ is separable over $K$.  Suppose $M$ is a $\D$-module containing an element $u$ such that $\m=\ann_R(u)$. Then the set $\{
\left(\SRgen{}{t}\right)u\mid \supp{(\vect t)}\in\Delta\}$ is linearly independent over $K$. 
\end{prop}

\begin{proof}
Without loss of generality, we replace $M$ by its $\D$-submodule generated by $u$.  Let $\bar K$ denote the algebraic closure of $K$ and set $\bar R = \bar K\otimes_KR=\bar K[\Delta]$, $\bar \m = \m \bar R$, $\bar \D = \bar K\otimes_K\D=D(\bar R;\bar K)$, $\bar M = \bar K\otimes_KM=\bar K\otimes_K \D u$.  Then $\bar M$ is a $\bar\D$-module and we have the injection $M\hookrightarrow\bar M$.  It suffices to show that $\{\left(\SRgen{}{t}\right)\bar u\mid\supp{(\vect t)}\in\Delta\}$ is linearly independent in $\bar M$ over $\bar K$, where $\bar u=1\otimes_K u$.  

Since $R/\m$ is separable over $K$, $\bar K\otimes_K(R/\m)$ is reduced.  Thus there exist maximal ideals $\bar{\m_1},\ldots,\bar{\m_s}\subset \bar R$ such that $\bar{\m}=\bigcap_{i=1}^s\bar{\m_i}$. Since $\bar K$ is algebraically closed, each $\bar{\m_i}$ is $K$-rational, and so by the Chinese remainder theorem,
\begin{equation}
\label{eq:ddmcrt}
\bar\D/\bar\D\bar\m\cong\bar\D\otimes_{\bar R}(\bar R/\bar\m)\cong \bar\D\otimes_{\bar R}\left(\bigoplus_{i=1}^s(\bar R/\bar{\m_i})\right)\cong\bigoplus_{i=1}^s(\bar \D/\bar\D\bar{\m_i}).
\end{equation}

Because $\bar\m=\ann_{\bar R}(\bar u)$ and $\bar M=\bar \D\bar u$, there is a $\bar \D$-module surjection 
\begin{subeqnarray*}
\varphi:\bar\D/\bar\D\bar\m  &\twoheadrightarrow \bar M \\
w &\mapsto w\cdot \bar u
\end{subeqnarray*}
We claim that $\varphi$ is injective.  Indeed, suppose there is $w\in \bar \D/\bar \D \bar\m$ such that $w\cdot \bar u=0$ and $w\neq 0$.  For $1\le i\le s$, let $w_i$ be the image of $w$ in $\bar\D/\bar\D\bar{\m_i}$ via \autoref{eq:ddmcrt}.  By \autoref{lem:ddm}(b), there is $f_i\in\bar{\m_i}$ such that $f_iw_i=1$. It follows that there is $f\in R$ such that $f$ is an $R$-multiple of $f_i$ and so $fw_i\in \bar R/\bar \m_i$ for each $i$, and such that $fw_i\neq 0$ for at least one $i$. Thus the tuple $(fw_1,\ldots,fw_s)\in\bigoplus_{i=1}^s(\bar R/\bar{\m_i})$ corresponds via the Chinese remainder theorem to an non-zero element $fw\in \bar R/\bar\m$. But $(fw)\cdot \bar u=f(w\cdot \bar u)=0$, so $fw\in\ann_{\bar R/\bar \m}(\bar u)=0$, a contradiction. 

By \autoref{lem:ddm}(a), $\{\left(\SRgen{}{t}\right)\mid \supp{(\vect t)}\in\Delta\}$ is linearly independent in each $\bar{\D}/\bar{\D}\m_i$, and so is also independent in $\bar{\D}/\bar{\D}\bar{\m}$.  Its image under the isomorphism $\varphi$ is $\{
\left(\SRgen{}{t}\right)\bar u\mid \supp{(\vect t)}\in\Delta\}$ and is linearly independent in $\bar M$. 
\end{proof}

Recall, from \autoref{sec:Stanley-Reisner}, the Hilbert function $H(R,j)$ gives the $K$-vector space dimension of the $j$th graded piece $R_j$.

\begin{prop}[cf. \cite{lyubeznik11}*{Proposition 3.1}]
\label{prop:filtdim}
Suppose $R=K[\Delta]$ is a $T$-space and $K$ is separable.  Let $M$ be a $\D$-module and suppose for $u\in M$, the annihilator $P=\ann_R(u)$ is a prime ideal in $R$. Then for all $i$, $\dim_K(\mathscr F_iu)\geq\sum_{j=0}^iH(R,j)$.
\end{prop}

\begin{proof}
We shall reduce to the hypotheses in \autoref{prop:ddmbasis}.  Let $Q$ denote the prime ideal in $S$ whose image in $R$ is $P$.  Let $h=\Ht Q$ and let $\mathcal K=\Quot(S/Q)$.  The transcendence degree of $\mathcal K$ over $K$ equals $n-h$ and by hypothesis $K$ is separable.  As in \cite{lyubeznik11}*{Proposition 3.1}, we wish to be able to use the assumption that $x_{h+1},\ldots,x_n$ are algebraically independent over $K$ in $\mathcal K$ and thus $\mathcal K$ is finite and separable over the field of rational functions $K'=K(x_{h+1},\dots,x_n)$; then consider $R'=K'\tns_SR$ .  Let $S'=K'\tns_SS\cong K'[x_1,\dots,x_h]$.  

Note $Q$ is a prime containing $I_{\Delta}$ and so contains a minimal prime of $I_{\Delta}$, call it $Q_0$.  By \autoref{thm:primesFacets}, $Q_0$ is generated by variables, and via a permutation of variables we may assert $Q_0=(x_1,\dots,x_l)$ with $l\leq h$.  If $l=h$ then $Q_0=Q$ and it is clear the image of $I_{\Delta}$ in $S'$ remains a monomial ideal.  Otherwise write $Q=Q_0+(f_1,\dots,f_m)$, where $m\geq h-l$, and the generators $f_i$ are polynomials in the variables $x_{l+1},\dots,x_n$ only.  Let $N\geq 1$ be an integer bounding the highest power of any $x_{l+1},\dots,x_n$ appearing in any of the $f_i$s.  We use the change of variables 
\begin{equation}
\label{eq:noethNorm}
x_i \mapsto x_i+x_{l+1}^{N^{i-h}}+x_{l+2}^{N^{i-h}}\cdots +x_h^{N^{i-h}}
\end{equation}
for $i=h+1,\dots,n$ to get the algebraic independence of $x_{h+1},\dots,x_n$.  

The functor $K'\tns_S(?)$ is equivalent to the operation of localizing at the multiplicative system consisting of polynomials over $K$ in the variables $x_{h+1},\dots,x_n$.  Therefore $R'=S'/I_{\Delta}S'$.  The image of $I_{\Delta}$ under \autoref{eq:noethNorm} may no longer be a monomial ideal.  However, apply the inverse of \autoref{eq:noethNorm} in $S'$; then, if such $x_i$ for $i=h+1,\dots,n$ appears in a minimal prime of $I_{\Delta}$, it is a unit.  Furthermore, note that $F=\{x_{h+1},\ldots,x_n\}$ is a face in $\Delta$ since $x_{h+1}\cdots x_n\notin I_{\Delta}$.  Thus $R'\cong K'[\Gamma]$, where 
\[
\Gamma=\{G\in\Delta \mid F\cup G\in \Delta,\, F\cap G=\emptyset\}
\]
(i.e., $\Gamma$ is the link of $F$ in $\Delta$).  Since $K[\Delta]$ is a $T$-space, so is $K'[\Gamma]$ (see \autoref{ex:link}).  By hypothesis, $\mathcal K\cong \Quot(S'/QS')= S'/QS'=R'/PR'$ is separable over $K'$.  

We now proceed as in \cite{lyubeznik11}*{Proposition 3.1}.  Let $\D'=D(R';K')$ and $M'=K'\tns_SR$.  From \autoref{thm:xdelx}, $\D'$ is a free $R'$-module on the products $x_1^{a_1}\partial_1^{t_1}\cdots x_h^{a_h}\partial_h^{t_h}=x_1^{a_1}\cdots x_h^{a_h}\partial_1^{t_1}\cdots \partial_h^{t_h}$ where $a_i=0$ if $t_i=0$, $a_i=1$ if $t_i\geq 1$, and $\supp(a_1,\dots,a_h)\in \Gamma\subset \Delta$.  Its action on $M$ naturally extends to $M'$ and so applying \autoref{prop:ddmbasis}, we get the $K'$-basis $\{
\left(\SRgen{}{t}\right)u'\mid \supp{(\vect t)}\in\Gamma\}\subset M'$, where $u'=1\tns u$.  It follows that the set 
\[
Y:=\{(x_{h+1}^{a_{h+1}}\cdots x_n^{a_n})x_1^{a_1}\partial_1^{t_1}\cdots x_h^{a_h}\partial_h^{t_h}u\mid \supp(\vect t)\subseteq \supp(\vect a)\in\Gamma\}
\] 
is linearly independent in $M$.  Let 
\[
Y_i=K\cdot \{x_{h+1}^{a_{h+1}}\cdots x_n^{a_n}x_1^{a_1}\partial_1^{t_1}\cdots x_h^{a_h}\partial_h^{t_h}u\in Y\mid a_1+\cdots+a_n+t_1+\cdots+ t_h\le i\}.
\] 
Then $Y_i\subseteq\mathscr F_iu$ and so $\dim_K\mathscr F_iu\geq |Y_i|$.  Since the set $Y_i$ is in bijection with the set of monomials not in the face ideal $I_{\Delta}$ of degree at most $i$, $|Y_i|=\sum_{j=0}^iH(R,j)$ and we get the result.
\end{proof}

\begin{thm}[cf. \cite{lyubeznik11}*{Theorem 3.5 and Corollary 3.3}]
\label{thm:finitelength}
Suppose $R=K[\Delta]$ is a $T$-space and let $r=\dim R$.  Then every holonomic $\D$-module $M$ has finite length in the category of $\D$-modules, bounded by a constant multiple of $r!$.
\end{thm}

\begin{proof}
Suppose $0=M_0\subsetneq M_1\subsetneq\cdots\subsetneq M_l=M$ is a $\D$-module filtration of $M$.  By the holonomicity hypothesis, $M$ also has a $K$-filtration $\mathscr G$ as in \autoref{dfn:holonomicity}, particularly, for all $i$, $\dim_K\mathscr G_i\leq Ci^r$ for some constant $C$. For each $s=1,\dots,l$, $\mathscr G$ induces a $K$-filtration $\mathscr G^s$ on $N_s:=M_s/M_{s-1}$, given by  
\[
\mathscr G^s_i=(\mathscr G_i\cap M_s)/(\mathscr G_i\cap M_{s-1}).
\]

Choose a prime $P_s\in\Ass_R(N_s)$. There exists $u_s\in N_s$ such that $P_s=\ann_R(u_s)$, and there exists $j_s$ such that $u_s\in\mathscr G^s_{j_s}$.  Since $\mathscr G^s$ is a $K$-filtration, we have $\mathscr F_{i-{j_s}}u_s\subseteq\mathscr G^s_i$ for all $i\geq j_s$.  By \autoref{prop:filtdim}, 
\[
\dim_K\mathscr G^s_i\geq \dim_K(\mathscr F_{i-j_s}u_s)\geq \sum_{j=0}^{i-j_s}H(R,j).
\] 
We also have, for each $i$, $\mathscr G_i=\oplus_{s=1}^lN_s$, because these are vector spaces over a field.  Therefore
\begin{equation}
\label{eq:filtpoly}
Ci^r\ge\dim_K\mathscr G_i=\sum_{s=1}^l \dim_K \mathscr G^s_i\ge\sum_{s=1}^l\sum_{j=0}^{i-j_s}H(R,j).
\end{equation}

The Hilbert function $H(R,j)$ equals the Hilbert polynomial for $R$ when $j>0$ and so the higher itereated Hilbert function $H_1(R,i-j_s)=\sum_{j=0}^{i-j_s}H(R,j)$ (see \cite{bruns+herzog}*{Section 4.1})  is a polynomial of degree $r$ in $i$ with leading coefficient at least $\frac{1}{r!}$. Thus \autoref{eq:filtpoly} yields $C\geq\frac{l}{r!}$, which implies $l\leq Cr!$ 
\end{proof}

\begin{thm}
\label{thm:Rholonomic}
Suppose $R=K[\Delta]$ is a $T$-space.  Then $R$ is a holonomic $\D$-module. 
\end{thm}

\begin{proof}
Put $r=\dim R$, and define an  $R$-filtration $\mathscr G= [ 0=\mathscr G_0\subset \mathscr G_1\subset \cdots\subset R]$, where $\mathscr G_i=\oplus_{j=0}^iR_j$ and for each $j\geq 0$, $R_j$ is the degree $j$ summand in the standard $\mathbb N$-grading for $R$.  It is clear that $\mathscr G$ is a $K$-filtration.  We claim it also satisfies the holonomicity property from \autoref{dfn:holonomicity}, namely, that there exists a constant $C$ such that for all $i$, $\dim_K{\mathscr G_i}\leq Ci^r$.  Since $R$ is the quotient of a polynomial ring over $K$ by a monomial ideal, $\dim_K{R_j}=H(R,j)$.  Each submodule $\mathscr G_i$ has $K$-length
\[
\dim_K{\mathscr G_i}=\sum_{j=0}^iH(R,j)=H_1(R,i), 
\]
given by the first higher iterated Hilbert function for $R$, a polynomial of degree $r$ in $i$ with leading coefficient at least $\frac{1}{r!}$.  
\end{proof}

\begin{thm}[cf. \cite{lyubeznik11}*{Corollary 3.6}]
\label{thm:Rfholonomic}
Suppose $R=K[\Delta]$ is a $T$-space.  Then for any $f\in R$, if $M$ is holonomic then so is $M_f$.  
\end{thm}

\begin{proof}
Let $d$ denote the degree of the polynomial $f\in R$ and let $r=\dim R$.  Suppose $\mathscr G_0\subset \mathscr G_1\subset\cdots$ is a $K$-filtration of $M$ with $\dim_K{\mathscr G_i}\leq Ci^r$ for some constant $C$.  Define
\[
\mathscr G_i'=K\cdot \left\{\frac{u}{f^i}\mid u\in \mathscr G_{i(d+1)}\right\},
\]
which gives a filtration, call it $\vect M'$, on $M_f$.   The meat of this proof is in showing $\mathscr G'$ is a $K$-filtration -- if $\mathscr G'$ is a $K$-filtration, then  
\[
\dim_K(\mathscr G_i')\leq \dim_K(\mathscr G_{i(d+1)})\leq C(i(d+1))^r
\]
and we can put $C'=C(d+1)^r$ to ensure $M_f$ satisfies the holonomicity condition.  

We first show  $\cup_i\mathscr G_i'=M_f$. We already have the inclusion ``$\subseteq$".  Choose $\frac{u}{f^w}\in M_f$, with $u\in \mathscr G_i$ for some $i$.  If $i\leq w(d+1)$, then $u\in \mathscr G_i\subseteq \mathscr G_{w(d+1)}$ and hence, by definition, $\frac{u}{f^w}\in \mathscr G_w'$.  On the other hand, if $i>w(d+1)$ then put $j=i-w(d+1)$.  Since $f^j\in\mathscr F_{jd}$, because $\deg f=d$, it follows that $f^ju\in \mathscr G_{jd+i}$.  Rewrite
\[
jd+i=jd+\left(j+w(d+1)\right)=(j+w)(d+1)
\]
to conclude $\frac{u}{f^w}=\frac{f^ju}{f^{j+w}}\in \mathscr G_{j+w}'\subset \cup_i\mathscr G_i'$.

We now show that $\mathscr F_i\mathscr G_j'\subset \mathscr G_{i+j}'$ for all $i,j$.  In other words, we show for $t\geq 0$, $x\partial^t(\mathscr G_j')\subset \mathscr G_{(t+1)+j}'$.  (Recall, $x\partial^t$ refers to any of the operators $x_i\partial_i^{t_i}$ in $\D$, with the subscript $i$ suppressed for readability.)  This is done by induction on $t$.  Let $\frac{u}{f^j}\in \mathscr G_j'$, with $u\in \mathscr G_{j(d+1)}$.  For the base case we just have $x\partial^0=x$.  Since $xf\in\mathscr F_{d+1}$, it follows that $(xf)(u)\in M_{j(d+1)+(d+1)}$.  Then $(x)\frac{u}{f^j}=\frac{(xf)(u)}{f^{j+1}}\in \mathscr G_{j+1}'$, as desired.

Suppose the induction hypothesis, that for $1\leq s\leq t$, $(x\partial^{t-s})\frac{u}{f^j}\in \mathscr G'_{(t-s+1)+j}$.  By \autoref{eq:quotientRule}, along with \autoref{rmk:Daction}, it suffices to show each of  
\begin{equation}
\label{eq:terms}
\frac{1}{f^j}x\partial^t(u)\quad
\text{and}\quad
\frac{1}{f^j}\frac{\partial^sf^j}{\partial x^s}\cdot x\partial^{t-s}\left(\frac{u}{f^j}\right)
\end{equation}
are in $\mathscr G_{(t+1)+j}'$.  Since $f\in\mathscr F_d$, $x\partial^t\in\mathscr F_{t+1}$, and $u\in \mathscr G_{j(d+1)}$, we have
\begin{subeqnarray*}
f^{t+1}\cdot x\partial^t &\in& \mathscr F_{d(t+1)+t+1}=\mathscr F_{(t+1)(d+1)}  \\
f^{t+1}\cdot x\partial^t(u) &\in& \mathscr G_{(t+1)(d+1)+j(d+1)}=\mathscr G_{(t+1+j)(d+1)}.
\end{subeqnarray*}
Thus
\[
\frac{1}{f^j}x\partial^t(u)=\frac{(f^{t+1}x\partial^t)(u)}{f^{t+1+j}}\in \mathscr G_{t+1+j}'.
\]

For the other term in \autoref{eq:terms} there exists, by induction hypothesis, an element $u_s\in \mathscr G_{(t-s+1+j)(d+1)}$ such that 
\[
x\partial^{t-s}\left(\frac{u}{f^j}\right)=\frac{u_s}{f^{t-s+1+j}}.
\]
We claim, by induction on $j$, that $f^{j-s}$ divides $\frac{\partial^sf^j}{\partial x^s}$.  But this is just the usual power rule in $D(S;K)$: when $j=0$, we have $\frac{\partial^s(1)}{\partial x^s}=0$, which certainly divides $f^{-s}$; then for any $j$, given the induction hypothesis, write $\frac{\partial^sf^{j-1}}{\partial x^s}$ as a quotient, $\frac{v_{s,j-1}}{f^{s-j+1}}$.  Using the usual product rule formula for higher-order derivatives, 
\begin{subeqnarray*}
\frac{\partial^t(f^{j-1}\cdot f)}{\partial x^t} 
	&=& \sum_{s=0}^t\frac{\partial^s f^{j-1}}{\partial x^s}\cdot \frac{\partial^{t-s}f}{\partial x^{t-s}} \\
	&=& \sum_{s=0}^t\frac{v_{s,j-1}}{f^{s-j+1}}\cdot \frac{v_{t-s,1}}{f^{(t-s)-1}} \\
	&=& \sum_{s=0}^t\frac{v_{s,j-1}v_{t-s,1}}{f^{s-j+1+t-s-1}},
\end{subeqnarray*}  
we can conclude $\frac{\partial^tf^j}{\partial x^t}$ divides $f^{j-t}$ and since $t$ was arbitrary we replace it with $s$.  We have
\begin{equation}
\label{eq:quotient}
\frac{1}{f^j}\frac{\partial^sf^j}{\partial x^s}\cdot \partial^{t-s}\left(\frac{u}{f^j}\right)=\frac{v_{s,j}u_s}{f^{t+j}}.
\end{equation}
The polynomial $\frac{\partial^sf^j}{\partial x^s}$ has degree $dj-s$, so the polynomial $v_{s,j}$ has degree $dj-s-d(j-s)=ds-s$.  Hence
\[
v_{s,j}u_s\in \mathscr G_{ds-s+(t-s+1+j)(d+1)}\subset \mathscr G_{(t+1+j)(d+1)},
\]
because 
\[
(t-s+1+j)(d+1)+s(d-1)<(t-s+1+j)(d+1)+s(d+1)=(t+1+j)(d+1).
\]
Thus $\frac{v_{s,j}u_s}{f^j}\in \mathscr G_{t+j}'$, which, by \autoref{eq:quotient}, means we are done. 
\end{proof}

In particular, $R_f$ is a holonomic $\D$-module. 

\begin{cor}
\label{cor:mainThm}
Suppose $R=K[\Delta]$ is a Stanley-Reisner ring over a field $K$, and its associated simplicial complex $\Delta$ is a $T$-space.  Let $I\subset R$ denote an ideal in $R$.  Then all local cohomology modules $H^j_I(R)$ have finitely many associated primes.
\end{cor}



\bibliographystyle{plain}
\bibliography{
wheelerbib}

\def\cprime{$'$}
\begin{bibdiv}
\begin{biblist}

\bib{bavula}{article}{
      author={Bavula, V.~V.},
       title={Dimension, multiplicity, holonomic modules, and an analogue of
  the inequality of {B}ernstein for rings of differential operators in prime
  characteristic},
        date={2009},
        ISSN={1088-4165},
     journal={Represent. Theory},
      volume={13},
       pages={182\ndash 227},
         url={http://dx.doi.org/10.1090/S1088-4165-09-00352-5},
      review={\MR{2506264 (2010h:16058)}},
}

\bib{bernstein+gelfand+gelfand}{article}{
      author={Bern{\v{s}}te{\u\i}n, I.~N.},
      author={Gel{\cprime}fand, I.~M.},
      author={Gel{\cprime}fand, S.~I.},
       title={Differential operators on a cubic cone},
        date={1972},
        ISSN={0042-1316},
     journal={Uspehi Mat. Nauk},
      volume={27},
      number={1(163)},
       pages={185\ndash 190},
      review={\MR{0385159}},
}

\bib{bblsz}{article}{
      author={Bhatt, Bhargav},
      author={Blickle, Manuel},
      author={Lyubeznik, Gennady},
      author={Singh, Anurag~K.},
      author={Zhang, Wenliang},
       title={Local cohomology modules of a smooth {$\Bbb{Z}$}-algebra have
  finitely many associated primes},
        date={2014},
        ISSN={0020-9910},
     journal={Invent. Math.},
      volume={197},
      number={3},
       pages={509\ndash 519},
         url={http://dx.doi.org/10.1007/s00222-013-0490-z},
      review={\MR{3251828}},
}

\bib{bogvad}{article}{
      author={B{\o}gvad, Rikard},
       title={Some results on {$\scr D$}-modules on {B}orel varieties in
  characteristic {$p>0$}},
        date={1995},
        ISSN={0021-8693},
     journal={J. Algebra},
      volume={173},
      number={3},
       pages={638\ndash 667},
         url={http://dx.doi.org/10.1006/jabr.1995.1107},
      review={\MR{1327873}},
}

\bib{brodmann+faghani}{article}{
      author={Brodmann, M.~P.},
      author={Faghani, A.~Lashgari},
       title={A finiteness result for associated primes of local cohomology
  modules},
        date={2000},
        ISSN={0002-9939},
     journal={Proc. Amer. Math. Soc.},
      volume={128},
      number={10},
       pages={2851\ndash 2853},
         url={http://dx.doi.org/10.1090/S0002-9939-00-05328-4},
      review={\MR{1664309}},
}

\bib{brumatti+simis}{article}{
      author={Brumatti, Paulo},
      author={Simis, Aron},
       title={The module of derivations of a {S}tanley-{R}eisner ring},
        date={1995},
        ISSN={0002-9939},
     journal={Proc. Amer. Math. Soc.},
      volume={123},
      number={5},
       pages={1309\ndash 1318},
         url={http://dx.doi.org/10.2307/2161115},
      review={\MR{1243162 (95f:13014)}},
}

\bib{bruns+herzog}{book}{
      author={Bruns, Winfried},
      author={Herzog, J{\"u}rgen},
       title={Cohen-{M}acaulay rings},
      series={Cambridge Studies in Advanced Mathematics},
   publisher={Cambridge University Press},
     address={Cambridge},
        date={1993},
      volume={39},
        ISBN={0-521-41068-1},
      review={\MR{1251956 (95h:13020)}},
}

\bib{hellus}{article}{
      author={Hellus, M.},
       title={On the set of associated primes of a local cohomology module},
        date={2001},
        ISSN={0021-8693},
     journal={J. Algebra},
      volume={237},
      number={1},
       pages={406\ndash 419},
         url={http://dx.doi.org/10.1006/jabr.2000.8580},
      review={\MR{1813886}},
}

\bib{huneke92}{incollection}{
      author={Huneke, Craig},
       title={Problems on local cohomology},
        date={1992},
   booktitle={Free resolutions in commutative algebra and algebraic geometry
  ({S}undance, {UT}, 1990)},
      series={Res. Notes Math.},
      volume={2},
   publisher={Jones and Bartlett, Boston, MA},
       pages={93\ndash 108},
      review={\MR{1165320 (93f:13010)}},
}

\bib{huneke+sharp}{article}{
      author={Huneke, Craig~L.},
      author={Sharp, Rodney~Y.},
       title={Bass numbers of local cohomology modules},
        date={1993},
        ISSN={0002-9947},
     journal={Trans. Amer. Math. Soc.},
      volume={339},
      number={2},
       pages={765\ndash 779},
         url={http://dx.doi.org/10.2307/2154297},
      review={\MR{1124167 (93m:13008)}},
}

\bib{24hours}{book}{
      author={{Iyengar}, Srikanth~B.},
      author={{Leuschke}, Graham~J.},
      author={{Leykin}, Anton},
      author={{Miller}, Claudia},
      author={{Miller}, Ezra},
      author={{Singh}, Anurag~K.},
      author={{Walther}, Uli},
      editor={{Cox}, David},
      editor={{Craig}, Walter},
      editor={{Ivanov}, N.~V.},
      editor={{Krantz}, Steven~G.},
       title={Twenty-four hours of local cohomology},
 institution={2005 AMS-IMS-SIAM Joint Summer Research Conference on},
     address={Providence, RI},
}

\bib{katzman}{article}{
      author={Katzman, Mordechai},
       title={An example of an infinite set of associated primes of a local
  cohomology module},
        date={2002},
        ISSN={0021-8693},
     journal={J. Algebra},
      volume={252},
      number={1},
       pages={161\ndash 166},
         url={http://dx.doi.org/10.1016/S0021-8693(02)00032-7},
      review={\MR{1922391 (2003h:13021)}},
}

\bib{lyubeznik93}{article}{
      author={Lyubeznik, Gennady},
       title={Finiteness properties of local cohomology modules (an application
  of {$D$}-modules to commutative algebra)},
        date={1993},
        ISSN={0020-9910},
     journal={Invent. Math.},
      volume={113},
      number={1},
       pages={41\ndash 55},
         url={http://dx.doi.org/10.1007/BF01244301},
      review={\MR{1223223 (94e:13032)}},
}

\bib{lyubeznik97}{article}{
      author={Lyubeznik, Gennady},
       title={{$F$}-modules: applications to local cohomology and {$D$}-modules
  in characteristic {$p>0$}},
        date={1997},
        ISSN={0075-4102},
     journal={J. Reine Angew. Math.},
      volume={491},
       pages={65\ndash 130},
         url={http://dx.doi.org/10.1515/crll.1997.491.65},
      review={\MR{1476089 (99c:13005)}},
}

\bib{lyubeznik00_1}{article}{
      author={Lyubeznik, Gennady},
       title={Finiteness properties of local cohomology modules: a
  characteristic-free approach},
        date={2000},
        ISSN={0022-4049},
     journal={J. Pure Appl. Algebra},
      volume={151},
      number={1},
       pages={43\ndash 50},
         url={http://dx.doi.org/10.1016/S0022-4049(99)00080-8},
      review={\MR{1770642 (2001g:13038)}},
}

\bib{lyubeznik00}{article}{
      author={Lyubeznik, Gennady},
       title={Finiteness properties of local cohomology modules for regular
  local rings of mixed characteristic: the unramified case},
        date={2000},
        ISSN={0092-7872},
     journal={Comm. Algebra},
      volume={28},
      number={12},
       pages={5867\ndash 5882},
         url={http://dx.doi.org/10.1080/00927870008827193},
        note={Special issue in honor of Robin Hartshorne},
      review={\MR{1808608 (2002b:13028)}},
}

\bib{lyubeznik11}{article}{
      author={Lyubeznik, Gennady},
       title={A characteristic-free proof of a basic result on {$\scr
  D$}-modules},
        date={2011},
        ISSN={0022-4049},
     journal={J. Pure Appl. Algebra},
      volume={215},
      number={8},
       pages={2019\ndash 2023},
         url={http://dx.doi.org/10.1016/j.jpaa.2010.11.012},
      review={\MR{2776441 (2012b:13068)}},
}

\bib{marley}{article}{
      author={Marley, Thomas},
       title={The associated primes of local cohomology modules over rings of
  small dimension},
        date={2001},
        ISSN={0025-2611},
     journal={Manuscripta Math.},
      volume={104},
      number={4},
       pages={519\ndash 525},
         url={http://dx.doi.org/10.1007/s002290170024},
      review={\MR{1836111}},
}

\bib{marley+vassilev}{article}{
      author={Marley, Thomas},
      author={Vassilev, Janet~C.},
       title={Cofiniteness and associated primes of local cohomology modules},
        date={2002},
        ISSN={0021-8693},
     journal={J. Algebra},
      volume={256},
      number={1},
       pages={180\ndash 193},
         url={http://dx.doi.org/10.1016/S0021-8693(02)00151-5},
      review={\MR{1936885}},
}

\bib{robbins}{article}{
      author={Robbins, Hannah},
       title={Associated primes of local cohomology after adjoining
  indeterminates},
        date={2014},
        ISSN={0022-4049},
     journal={J. Pure Appl. Algebra},
      volume={218},
      number={11},
       pages={2072\ndash 2080},
         url={http://dx.doi.org/10.1016/j.jpaa.2014.03.005},
      review={\MR{3209808}},
}

\bib{singh}{article}{
      author={Singh, Anurag~K.},
       title={{$p$}-torsion elements in local cohomology modules},
        date={2000},
        ISSN={1073-2780},
     journal={Math. Res. Lett.},
      volume={7},
      number={2-3},
       pages={165\ndash 176},
         url={http://dx.doi.org/10.4310/MRL.2000.v7.n2.a3},
      review={\MR{1764314 (2001g:13039)}},
}

\bib{singh+swanson}{article}{
      author={Singh, Anurag~K.},
      author={Swanson, Irena},
       title={Associated primes of local cohomology modules and of {F}robenius
  powers},
        date={2004},
        ISSN={1073-7928},
     journal={Int. Math. Res. Not.},
      number={33},
       pages={1703\ndash 1733},
         url={http://dx.doi.org/10.1155/S1073792804133424},
      review={\MR{2058025 (2005d:13030)}},
}

\bib{takagi+takahashi}{article}{
      author={Takagi, Shunsuke},
      author={Takahashi, Ryo},
       title={{$D$}-modules over rings with finite {$F$}-representation type},
        date={2008},
        ISSN={1073-2780},
     journal={Math. Res. Lett.},
      volume={15},
      number={3},
       pages={563\ndash 581},
         url={http://dx.doi.org/10.4310/MRL.2008.v15.n3.a15},
      review={\MR{2407232 (2009e:13003)}},
}

\bib{traves}{article}{
      author={Traves, William~N.},
       title={Differential operators on monomial rings},
        date={1999},
        ISSN={0022-4049},
     journal={J. Pure Appl. Algebra},
      volume={136},
      number={2},
       pages={183\ndash 197},
         url={http://dx.doi.org/10.1016/S0022-4049(97)00170-9},
      review={\MR{1674776 (2000i:13028)}},
}

\bib{tripp}{article}{
      author={Tripp, J.~R.},
       title={Differential operators on {S}tanley-{R}eisner rings},
        date={1997},
        ISSN={0002-9947},
     journal={Trans. Amer. Math. Soc.},
      volume={349},
      number={6},
       pages={2507\ndash 2523},
         url={http://dx.doi.org/10.1090/S0002-9947-97-01749-2},
      review={\MR{1376559 (97h:13020)}},
}

\end{biblist}
\end{bibdiv}

\end{document}